\documentclass[reqno,a4paper]{amsart}
\usepackage{amsmath}
\usepackage{amssymb,amsfonts}
\usepackage{amsthm}
\usepackage{enumerate}
\usepackage[mathscr]{eucal}
\usepackage{eqlist}
\usepackage{xcolor}

\theoremstyle{plain}
\newtheorem{theorem}{Theorem}[section]
\newtheorem{prop}[theorem]{Proposition}
\newtheorem{lemma}{Lemma}[section]

\theoremstyle{definition}

\theoremstyle{remark}

\numberwithin{equation}{section}

\begin{document}
	
	\title[on the exponential Diophantine eq. $(n-1)^{x}+(n+2)^{y}=n^{z}$]
	{on the exponential Diophantine equation $(n-1)^{x}+(n+2)^{y}=n^{z}$}
	
	\author{Hairong Bai, El\.{I}f K{\i}z{\i}ldere, G\"{o}khan Soydan and Pingzhi Yuan}
	
	\address{{\bf Hairong Bai} \\
		School of Mathematics\\
		South China Normal University\\
		Guangzhou 510631, China}
	\email{584846139@qq.com}
	
	\address{{\bf Elif K{\i}z{\i}ldere}\\
		Department of Mathematics \\
		Bursa Uluda\u{g} University\\
		16059 Bursa, Turkey}
	\email{elfkzldre@gmail.com}
	
	\address{{\bf G\"{o}khan Soydan} \\
		Department of Mathematics \\
		Bursa Uluda\u{g} University\\
		16059 Bursa, Turkey}
	\email{gsoydan@uludag.edu.tr }
	
	\address{{\bf Pingzhi Yuan} \\
		School of Mathematics\\
		South China Normal University\\
		Guangzhou 510631, China}
	\email{mcsypz@mail.sysu.edu.cn}
	
	\newcommand{\acr}{\newline\indent}

	\thanks{}
	
	
	\begin{abstract}
		Suppose that $n$ is a positive integer. In this paper, we show that the exponential Diophantine equation
		$$(n-1)^{x}+(n+2)^{y}=n^{z},  \, \, n\geq 2, \, \, xyz\neq 0$$
		has only the positive integer solutions $(n,x,y,z)=(3,2,1,2), (3,1,2,3)$. The main tools on the proofs are Baker's theory and Bilu-Hanrot-Voutier's result on primitive divisors of Lucas numbers.
	\end{abstract}

	\maketitle

\section{Introduction and the main result}\label{sec:1}
As usual, we denote the set of all integers by $\mathbb{Z}$, the set of positive integers by $\mathbb{N}$. Suppose that $a, b, c$ are pair-wise coprime positive integers. Then we call the
equation
\begin{equation}\label{11.1}
a^x+b^y=c^z, \ \ x, y, z \in \mathbb{N}
\end{equation}
as an exponential Diophantine equation. Many authors have studied the above equation for given $a,b,c\in\mathbb{N}.$

Equation \eqref{11.1} was first considered by Mahler, \cite{Ma}. He proved that the equation \eqref{11.1} has finitely many solutions with $a, b, c >1$. Since his method was based on a
$p$-adic generalisation of Thue-Siegel method, it was ineffective as it provided no bounds for the size of possible solutions. Later, Gel'fond, \cite{Ge}, gave an effective result for
solutions of \eqref{11.1}. His method was based on Baker's theory which uses linear forms in the logarithms of algebraic numbers.

Using elementary number theory methods such as congruences, Jacobi symbol and standard divisibility arguments in algebraic number theory involving ideals in quadratic (or cubic) number
fields, the complete solutions of the equation \eqref{11.1} where $a, b, c$ are distinct primes $\leq 17$ were determined by some authors (see  \cite{DYL18}, \cite{Ha}, \cite{Na} and
\cite{U}).

Consider the equation \eqref{11.1} for Pythagorean numbers $a$, $b$, $c$. A famous unsolved problem concerning the exponential Diophantine equation \eqref{11.1} was suggested by
Je\'{s}manowicz, \cite{J}. He has conjectured that the unique solution in positive integers of equation \eqref{11.1} is only $(x, y, z)=(2, 2, 2)$, where $a, b, c$ are satisfying
$a^2+b^2=c^2$, i.e. they are Pythagorean triples. This conjecture have been solved for many special cases. Different conjectures concerning equation \eqref{11.1} were identified and
discussed.
One of these conjectures which is an extension of Je\'{s}manowicz conjecture, was suggested by Terai.

Terai conjectured that the unique solution of the Diophantine equation \eqref{11.1} is $(x,y,z)=(u,v,w)$ except for some $(a,b,c)$ where $a,b,c,u,v,w \in \mathbb{N}$ are fixed, $u, v, w \geq
2$ and $\gcd(a,b)=1$ and satisfying $a^u+b^v=c^w$ (see \cite{Cao}, \cite{HY18}, \cite{Le}, \cite{Mi1}, \cite{Mi2}, \cite{T1}, \cite{T2}, \cite{YH18}). The correctness of this conjecture for
many special cases has been proved.
Nevertheless, it has been still unsolved. Recently, a survey paper on the conjectures of Je\'{s}manowicz and Terai has been published by Soydan, Demirci, Cang\"{u}l and Togb\'{e} (see
\cite{SCDT} for the details about these conjectures).

Now we take the exponential Diophantine equation
\begin{equation}\label{1.2}
(tb-1)^x+b^y=(tb+1)^z
\end{equation}
with $t \geq1$. Clearly, it suffices to consider the case where $b$ is even. So, we see that equation \eqref{1.2} has the following trivial solutions:
\begin{equation*}
(x,y,z)= \left\{\def\arraystretch{1.2}
\begin{array}{@{}c@{\quad}l@{}}
(i,1,1): i\geq1 ,\,(j,3,2): j\geq1 & \text{if $b=2$ and $t=1$,}\\
(2,k+1,2) & \text{if $t=b^k/4$ with $k\geq1$,}\\
(1,1,1) & \text{if $b=2$,}\\
(1,13,2) & \text{if $b=2$ and $t=45$,} \\
\end{array}\right.
\end{equation*}
In \cite{He-To}, He and Togb\'{e} solved equation \eqref{1.2} for $t=1$. After the work of \cite{He-To}, Miyazaki and Togb\'{e} \cite{Mi-To} extended their result by proving equation
\eqref{1.2} and they proved that equation \eqref{1.2} has no non-trivial solution when $t>1$ is odd. Lastly, Miyazaki, Togb\'{e} and Yuan \cite{Mi-To-Yu} showed that equation \eqref{1.2} has
no non-trivial solution for any positive integer $t$, completing the study of \eqref{1.2}.

In this paper, we consider the exponential Diophantine equation
\begin{equation}\label{1.1}
(n-1)^{x}+(n+2)^{y}=n^{z}
\end{equation}
with $n\geq 2$, $xyz\neq 0$.

Indeed, the equation \eqref{1.1} is a generalisation of the equations $2^x+5^y=3^z$ and $4^x+7^y=5^z$ considered by Nagell \cite{Na}. Our main result is the
following theorem.
\begin{theorem} \label{theo:1.1}
Let $n$ be a positive integer. Then the equation \eqref{1.1} has only positive integer solutions $(n,x,y,z)=(3,2,1,2), (3,1,2,3).$
\end{theorem}

%################################%
\section{Auxiliary results}
%################################%
In this section, we give some useful lemmas which will be used to prove our main result.

\begin{prop}\label{Hua} Let $D$ be a positive integer. We have
	$$h(-4D)<\frac{4\sqrt{D}}{\pi}\log(2e\sqrt{D}).$$
	Further, $ h(-4D)$ is equal to the class number of the unique quadratic order of discriminant $-4D$.
\end{prop}

\begin{proof} The upper bound for $ h(-4D)$ follows from Theorems 11.4.3, 12.10.1 and 12.14.3 of Hua \cite{KH82}, while the last assertion is contained in Definition 5.2.7 of \cite{HC93}.
\end{proof}

A Lucas pair is a pair $(\alpha, \, \beta)$ of algebraic integers such that $\alpha+\beta$ and $\alpha\beta$ are non-zero coprime rational
integers and $\alpha/\beta$ is not a root of unity. Given a Lucas pair 	$(\alpha, \, \beta)$, one defines the corresponding sequence of
Lucas numbers by $$u_n=u_n(\alpha,\beta)=\frac{\alpha^n-\beta^n}{\alpha-\beta},\quad n=0, \, 1, \, 2, \ldots.$$

Let $(\alpha, \, \beta)$ be a Lucas pair. We recall that a prime number $p$ is a primitive divisor of the
Lucas number $u_n(\alpha, \, \beta)$ if $p$ divides $u_n$ but does not divide 	$(\alpha-\beta)^2u_1u_2\cdots u_{n-1}$. We say that a Lucas sequence is an $n-$defective Lucas sequence if
$u_n$ has no primitive divisor. The key argument for the proofs in this section is the definitive result obtained by Bilu, Hanrot and Voutier \cite{BHV01}.

\begin{lemma}{\rm(\cite{BHV01})}\label{le21} For any integer $n>30$, every Lucas sequence is non n-defective. Further, for any positive integer $n<30$, all n-defective Lucas sequence are
	explicitly determined.
\end{lemma}

\begin{lemma}{\rm(\cite{Vo95})}\label{le22} For $4<n\leq{30}$, and $n$ is odd, the following gives a complete list, up to the sign of $\alpha$ and $\beta$, of all Lucas sequences whose n-th
	element has no primitive divisor.
	
	$n=5,\quad  (\alpha, \beta)=(\pm{\frac{1\pm\sqrt{5}}{2}},\,\,\pm{\frac{1\mp\sqrt{5}}{2}}), \,\,\,\,(\pm{\frac{1\pm\sqrt{-7}}{2}},\,\pm{\frac{1\mp\sqrt{-7}}{2}})$,
	$$\qquad \qquad \qquad \qquad (\pm{\frac{1\pm\sqrt{-15}}{2}},\,\ \pm{\frac{1\mp\sqrt{-15}}{2}}), 	
	(\pm(6\pm\sqrt{-19}),\,\ \pm(6\mp\sqrt{-19})),$$ $$\qquad \qquad \qquad \qquad (\pm(1\pm\sqrt{-10}),\,\,\,\,\pm(1\mp\sqrt{-10})), 	
	(\pm{\frac{1\pm\sqrt{-11}}{2}},\,\pm{\frac{1\mp\sqrt{-11}}{2}}),$$ $\qquad \qquad \qquad \qquad(\pm(6\pm\sqrt{-341}),\pm(6\mp\sqrt{-341}));$
	
	$n=7, \quad (\alpha, \beta)=(\pm{\frac{1\pm\sqrt{-7}}{2}},\pm{\frac{1\mp\sqrt{-7}}{2}}),\,\,\,\, (\pm{\frac{1\pm\sqrt{-19}}{2}},\,\,\pm{\frac{1\mp\sqrt{-19}}{2}})$;
	
	$n=13,\quad (\alpha, \beta)=(\pm{\frac{1\pm\sqrt{-7}}{2}},\,\,
	\pm{\frac{1\mp\sqrt{-7}}{2}})$.
\end{lemma}

Suppose $c\in\{1,2,4\}$, $\lambda^2=c$ and $2\nmid k$ if $c\in\{1,2\}$. Then we have
\begin{lemma}\label{Represent}{\rm (\cite[Lemma 1]{BS01} or \cite[Corollary 3.1]{Y04} )} Let $D_1$ and $D_2$ be coprime positive integers and let $k\ge 2$ be an integer coprime with $D_1D_2$.
	
	$(i)$ Let $D_1D_2\not\in\{1,3\}$.
	The solutions of equation
	\begin{equation}\label{Quad} D_1 X^2+ D_2Y^2=\lambda^2k^Z, \quad  X, \, Y, \, Z\in\mathbb{Z}, \,\, \gcd(X, \, Y)=1, \,\,  Z>0
	\end{equation}
	can be put into at most $2^{\omega(k)-1}$ classes, where $\omega(k)$ denotes the number of distinct prime divisors of $k$. Further, in each such class $S$ there is a unique solution $(X_1, \, Y_1,  Z_1)$ such that $X_1>0,Y_1>0$ and $Z_1$ is minimal among the
	solutions of $S$. This minimal solution satisfies $Z_1$ divides $h(-4D)$ if $D_1=1$ or $D_2=1$ and $2Z_1$ divides $h(-4D)$
	otherwise. Moreover, every solution $(X, \, Y, \, Z)$ of \eqref{Quad} belonging to $S$ can be expressed as
	\begin{equation}\label{Quad1}Z= Z_1t, \quad \left(\frac{X\sqrt{D_1}+ Y\sqrt{-D_2}}{\lambda}\right)=\lambda_1\left(\frac{X_1\sqrt{D_1}+
		\lambda_2Y_1\sqrt{-D_2}}{\lambda}\right)^t\end{equation}
	where $t\ge1$ is an integer, $\lambda_1\in\{1, \, -1, \, i, \, -i\}$ and $\lambda_2\in\{1, \, -1\}$. If $\lambda=\sqrt{2}$, then $t$ is odd. Further $\lambda_1\in\{1, \, -1\}$ if $D_2\ne1$
	or $t$ is odd and $\lambda_1\in\{i, \, -i\}$ if $D_2=1$ and $t$ is even.
	
	$(ii)$ Let $D_1D_2\in\{1,3\}$. Then the solutions of \eqref{Quad} can be put into at most $2^{\omega(k)-1}$ classes. Further, in each such class $S$ there is a unique solution $(X_1, \, Y_1,
	\, 1)$ with $X_1>0,Y_1>0$ such that every solution 		$(X, \, Y, \, Z)$ of \eqref{Quad} belonging to $S$ satisfies \eqref{Quad1}
	with $Z_1=1$, $\lambda_1\in\{1, \, -1, \, i, \, -i\}$ if $D_1D_2=1$, $$\lambda_1\in\{-1, 1,\, i, \, -i, \,
	\frac{1+i\sqrt{3}}{2}, \, \frac{1-i\sqrt{3}}{2}, \, \frac{-1+i\sqrt{3}}{2}, \, \frac{-1-i\sqrt{3}}{2} \}$$
	and $\lambda_2\in\{-1, 1\}$ if $D_1D_2=3$.
\end{lemma}

Now, in order to obtain an upper bound for solutions, we quote a result on lower bounds for linear forms in the logarithms of two algebraic numbers. To do this, we first present some notations. Let $\varphi_1$ and $\varphi_2$  be real algebraic numbers with $|\varphi_1|\geq1$ and $|\varphi_2|\geq1$. We consider the linear form
\begin{equation*}
\Omega=c_2 \log \varphi_2-c_1 \log \varphi_1,
\end{equation*}
where $\log \varphi_1$,  $\log \varphi_2$ are any determinations of the logarithms of  $\varphi_1$,  $\varphi_2$ respectively, and $c_1$, $c_2$ are positive integers. Let $\varphi$ be any non-zero algebraic number with minimal polynomial over $\mathbb{Z}$ is  $a_{0} \prod_{j=1}^{d} (X-\varphi^{(j)})$ which is of degree $d$ over $\mathbb{Q}$. We denote by
\begin{equation*}
h(\varphi)=\frac{1}{d} \Big( \log|a_{0}|+ \sum_{j=1}^{d} \log\max\{1,|\varphi^{(j)}|\}\Big),
\end{equation*}
the absolute logarithmic height of $\varphi$. Suppose that $B_1$ and $B_2$ are real numbers $\geq1$ such that
\[
\log B_{j} \geq \max \Big\{ h(\varphi_{j}), \frac{|\log \varphi_{j}|}{D}, \frac{1}{D} \Big\}\ \ \ (j=1,2),
\]
where the number field $\mathbb{Q}(\varphi_{1}, \varphi_{2})$ over $\mathbb{Q}$ has degree $D$. Define
\begin{equation}\label{La}
d'= \dfrac{c_1}{D\log B_2}+\dfrac{c_2}{D\log B_1}.
\end{equation}

The following lemma is an immediate consequence of \cite[Corollary 2]{La}. Here we take $m=10$ and $C_2=25.2$ (according to the notation of the paper).

\begin{lemma}{\rm (\cite{La})}\label{Lemma2.3}
	Suppose that $\Omega$ is given as above and $\varphi_1>1$ and $\varphi_2>1$ are multiplicatively independent. Then
	\[
	\log |\Omega| \geq -25.2 D^4 \max \{\log d'+0.38, 10/D\})^2 \log B_1 \log B_2.
	\]
\end{lemma}

%################################%
\section{Proof of Theorem \ref{theo:1.1}}
%################################%
When $n=2$, the equation \eqref{1.1} becomes $2^z-4^y=1$. It is clear that this equation has no positive integer solutions. So, in this section we suppose that $n>2$ and
we distinguish four subcases: $n>64$, $n \ge 7$ and $z\ge2n$, $7\le n \le 64$ and $z<2n$, $2<n<7$ separately.

\subsection{The case $n>64$}\label{sec:3.1}

Suppose that $n>64$. If $(n-1)^x\ge\frac{1}{2}n^z$, then $x\ge z$. First assume that $x>z$. Then one has
$$1>\left(1-\frac{1}{n}\right)^z\cdot(n-1)^{x-z}>\frac{(n-1)^{x-z}}{e^{z/n}},$$ since $\left(1-\frac{1}{n}\right)^z>e^{-z/n}$. It follows that $e^{z/n}>(n-1)^{x-z}>63$, which implies that
$z>4n$.

Next consider the case where $x=z$. Then one gets $(n+2)^y>\frac{1}{n}\cdot n^z$. Note that $y<z$, and we have
$$e^{2y/n}>\left(1+\frac{2}{n}\right)^{n/2\times2y/n}>\frac{1}{n}\cdot n^{z-y}.$$
If $z-y>1$, then the previous inequality implies that $e^{2y/n}>n\ge 64$, so $y>2n$.

If $(n-1)^x<\frac{1}{2}n^z$, then $(n+2)^y>\frac{1}{2}\cdot n^z$, hence $$e^{2y/n}>\left(1+\frac{2}{n}\right)^{n/2\times2y/n}>\frac{1}{2}\cdot n^{z-y}.$$
If $z-y>1$, then using the above inequality we get that $e^{2y/n}>n^2/2\ge 64$, so $y>2n$. If $z-y=1$, similarly, we have $e^{2y/n}>\frac{n}{2}\ge 32$, so $z>y>n$. Therefore we have proved that
$$x>z>4n \quad \mbox{ or } \quad z>y>n \quad \mbox{ or } \quad x=z=y+1.$$

Now from $(n-1)^{y+1}+(n+2)^y=n^{y+1}$ and $n>64$, we get $y\le3$ or $y>n$. If $(n-1)^{y+1}+(n+2)^y=n^{y+1}$ and $y\le3$, then a simple computation shows that $y=1$ and $n=3$. Consequently,
we have shown that $$x>z>4n \quad \mbox{ or } \quad z>y>n$$ when $n>64$.

\quad

\textbf{$(i)$ The case $y$ is even}

\quad

Assume that $y$ is even. Since $z>n$ and $n>64$, by Proposition \ref{Hua} we have $z>n>h(\mathbb{Q}(\sqrt{-4(n-1)}))$, where $h(\mathbb{Q}(\sqrt{-d}))$ denotes the class number for the quadratic
field $\mathbb{Q}(\sqrt{-d})$. Therefore by Lemma \ref{Represent} there exist integers $x_0, y_0, z_0$ such that $z_0\mid h(\mathbb{Q}(\sqrt{-4(n-1)}))$ and
$$(n-1)x_0^2+y_0^2=n^{z_0}, \quad 2\nmid x,$$
or
$$x_0^2+y_0^2=n^{z_0}, \quad 2\mid x,$$
and we have
\begin{equation}\label{A1}
(n-1)^{\frac{x-1}{2}}\sqrt{-(n-1)}+(n+2)^{\frac{y}{2}}=\pm(x_0\sqrt{-(n-1)}+y_0)^{z/z_0},
\end{equation}
or
\begin{equation}\label{A2}
(n-1)^{\frac{x}{2}}\sqrt{-1}+(n+2)^{\frac{y}{2}}=\pm(x_0\sqrt{-1}+y_0)^{z/z_0},
\end{equation}
which is impossible by Lemma \ref{le21} and \ref{le22} when $z/z_0\ge5$.

If $z/z_0=2$, then the argument for $z/z_0=4$ is the same. So, we can omit it.

First consider the equation \eqref{A1} with $z/z_0=2$. Then one has\\ $(n-1)^{\frac{x-1}{2}}\sqrt{-(n-1)}+(n+2)^{\frac{y}{2}}=\pm(x_0\sqrt{-(n-1)}+y_0)^2$, and so
$$2x_0y_0=\pm (n-1)^{\frac{x-1}{2}}.$$
Since $x_0^2(n-1)+y_0^2=n^{z_0}$, we have $y_0=\pm1$ and $2x_0=(n-1)^{\frac{x-1}{2}}$, it follows that $(n-1)^x+4=4n^{z_0}$. Taking modulo 4 and recall that $x$ is odd, we have $n=3$.
We know that the equation $2^x+5^y=3^z$ has only the solution $(x, y, z)=(1, 2, 3)$ (see \cite[Theorem 3]{Na}) when $2\nmid x$ and $2\mid y$. But, this is not a desired solution.

Next consider the equation \eqref{A1} with $z/z_0=3$. Then one gets\\ $(n-1)^{\frac{x-1}{2}}\sqrt{-(n-1)}+(n+2)^{\frac{y}{2}}=\pm(x_0\sqrt{-(n-1)}+y_0)^3$, and hence
$$x_0(3y_0^2-x_0^2(n-1))=\pm (n-1)^{\frac{x-1}{2}}, \quad y_0(y_0^2-3x_0^2(n-1))=\pm (n+2)^{\frac{y}{2}}.$$
Note that $\gcd(3, n-1)=1$, so we derive that $x_0=\pm (n-1)^{\frac{x-1}{2}}$ and $y_0=\pm (n+2)^{\frac{y}{2}}$, and $x_0^2(n-1)+y_0^2=n^{z_0}$, which is impossible.

Now first consider the equation \eqref{A2} with $z/z_0=2$. Then one obtains\\ $(n-1)^{\frac{x}{2}}\sqrt{-1}+(n+2)^{\frac{y}{2}}=\pm(x_0\sqrt{-1}+y_0)^2$, and so
$$2x_0y_0=\pm (n-1)^{\frac{x-1}{2}}.$$
Since $x_0^2(n-1)+y_0^2=n^{z_0}$, we have $y_0=\pm1$ and $2x_0=(n-1)^{\frac{x}{2}}$, it follows that $(n-1)^x+4=4n^{z_0}$. Taking modulo 4 and recall that $x$ is even, we have $n=5$.
We know that the equation $4^x+7^y=5^z$ has no solution $(x, y, z)$ (see \cite[Theorem 9]{Na}) when $2\nmid x$ and $2\mid y$.

Finally consider the equation \eqref{A2} with $z/z_0=3$. Then one has\\ $(n-1)^{\frac{x}{2}}\sqrt{-1}+(n+2)^{\frac{y}{2}}=\pm(x_0\sqrt{-1}+y_0)^3$, and hence
$$x_0(3y_0^2-x_0^2(n-1))=\pm (n-1)^{\frac{x}{2}}, \quad y_0(y_0^2-3x_0^2(n-1))=\pm (n+2)^{\frac{y}{2}}.$$
Note that $\gcd(3, n-1)=1$, so we derive that $x_0=\pm (n-1)^{\frac{x}{2}}$ and $y_0=\pm (n+2)^{\frac{y}{2}}$, and $x_0^2(n-1)+y_0^2=n^{z_0}$,  which is impossible. So, this proof of the
case $y$ is even is completed.

\quad

\textbf{$(ii)$ The case $y$ is odd and $x$ is even}

\quad

Suppose that $y$ is odd and $x$ is even. Since $z>n$ and $n>64$, by Proposition \ref{Hua} one has $z>n>h(\mathbb{Q}(\sqrt{-4(n+2)}))$. Therefore by Lemma \ref{Represent} there exist integers $x_0, y_0,
z_0\mid h(\mathbb{Q}(\sqrt{-4(n+2)}))$ such that
$$x_0^2+(n+2)y_0^2=n^{z_0}, $$
and one has
$$(n-1)^{\frac{x}{2}}+(n+2)^{\frac{y-1}{2}}\sqrt{-(n+2)}=\pm(x_0+y_0\sqrt{-(n+2)})^{z/z_0},$$
which is impossible by Lemma \ref{le21} and \ref{le22} when $z/z_0\ge5$. It is easy to check that $z/z_0=2, 3, 4$ is also impossible. Its proof is similar to the case $(i).$

\quad

\textbf{$(iii)$ The case $x$ and $y$ are odd}

\quad

Assume that $x$ and $y$ are odd. If $z$ is even, then by taking modulo $n+1$ for equation \eqref{1.1}, one gets
$$(-2)^x\equiv0\pmod{n+1},$$
which implies that $n+1=2^t$,  $t\le x$. Now equation \eqref{1.1} becomes
\begin{equation}\label{xx}
(2^t-2)^x+(2^t+1)^y=(2^t-1)^z,\,\, 2\mid z.
\end{equation}
Consider the equation \eqref{xx} modulo 3. If $t$ is even, then one obtains
$(-1)^x+(-1)^y\equiv 0 \pmod{3}$, which is $-2\equiv 0 \pmod{3}$. This is a contradiction. If $t$ is odd, then one gets $0 \equiv 1 \pmod{3}$ which is also contradiction. So, the proof is completed for $(iii)$.

\quad

\textbf{$(iv)$ The case $x$, $y$ and $z$ are odd}

\quad

Suppose that $x$, $y$ and $z$ are odd. Taking modulo $n$ for equation \eqref{1.1}, one has
$$\left(\frac{2}{n}\right)=1,$$
which implies that $n\equiv 1, 7\pmod{8}$.

First consider the case $n\equiv1\pmod{8}$. If $n\equiv1\pmod{8}$,   by taking modulo $8$  for equation \eqref{1.1}, one gets
$$3^y\equiv1\pmod{8},$$ which implies that $2\mid y$. This is a contradiction.

Next consider the case $n\equiv7\pmod{8}$. Here we have two subcases:

\begin{itemize}
\item The case $x\ge 3$.

If $n\equiv7\pmod{8}$, by taking modulo $8$  for equation \eqref{1.1}, one has
$$6^x+1^y\equiv(-1)^z\pmod{8},$$ which implies that $2\mid z$, a contradiction where $x\ge 3$ is odd.

\item The case $x=1$.

Now we have to consider the equation \eqref{1.1} where $n\equiv 7 \pmod{8}$, $y$ and $z$ are odd, $x=1$. We have $x>z>4n$ or $z>y>n$ when $n>64$. It is clear that $x=1$ is impossible where $x>z>4n$. So we consider the case $z>y>n>64$.

The equation \eqref{1.1} is reduced to solving the equation
\begin{equation}\label{eq3.34}
n-1+(n+2)^y=n^z.
\end{equation}
Since $z>y>n$ and $n \equiv 7 \pmod 8$, we have $y\geq73$ and $n\geq 71$. If $z\geq2y$, then
\[
n-1=n^z-(n+2)^y\geq n^{2y}-(n+2)^y> n^2-(n+2)
\]
\[=n^2-n-2,
\]
which is impossible. So, $2y>z$.

Let
\[
\Omega=z\log n-y\log(n+2).
\]
From \eqref{eq3.34}, we find that $e^{\Omega}-1=\frac{n-1}{(n+2)^y}$. Since $n\geq71$ and $\log(1+t)<t$ for $t>0$, one gets
\begin{equation}\label{eq3.35}
0< \Omega < \dfrac{1}{1.04(n+2)^{y-1}}.
\end{equation}
From the definition $\Omega$ and $2y>z$, we get
\[
\frac{2y-1}{y}\geq\frac{z}{y}>\frac{\log(n+2)}{\log n}.
\]
So, we have
\begin{equation}\label{eq3.36}
y> \dfrac{\log n}{\log(\frac{n^2}{n+2})}> (n-70.99)\log n.
\end{equation}

In \eqref{La} taking
\[
D=1, \ \ c_1=y, \ \ c_2=z, \ \ B_1=n+2, \ \ B_2=n
\]
and using \eqref{eq3.35}, we obtain
\[
\dfrac{z}{\log(n+2)}-\dfrac{y}{\log n}<\dfrac{1}{1.04 (n+2)^{y-1}\log n\log(n+2)}
<0.36\cdot10^{-135}.
\]
The above inequality gives

\[d'<\frac{2y}{\log n}+0.36\cdot10^{-135},\]
where
$d'=\frac{z}{\log(n+2)}+\frac{y}{\log n}.
$

If $\log d'+0.38>10$, then we get
\begin{equation}\label{eq3.37}
\log |\Omega| \geq -25.2(\log(\dfrac{2y}{\log n}+0.36\cdot10^{-135})+0.38)^2\log n\log(n+2).
\end{equation}
Using \eqref{eq3.35}, we have
\begin{equation}\label{eq3.38}
\log |\Omega| < -(y-1)\log (n+2)-0.039.
\end{equation}
Combining \eqref{eq3.37} and \eqref{eq3.38} gives
\begin{equation}\label{eq3.39}
\dfrac{y}{\log n}<\dfrac{-0.039+\log (n+2)}{\log n\cdot\log(n+2)}+25.2(\log(\dfrac{2y}{\log n}+0.36\cdot10^{-135})+0.38)^2.
\end{equation}
So,
\[
\dfrac{y}{\log n}<0.24+25.2(\log(\dfrac{y}{\log n}+0.18\cdot10^{-135})+1.08)^2.
\]
As a result of the above inequality, one obtains $y<1870\log n$. Then $d'<2y/\log n+0.36\cdot10^{-135}<3741$ so that $\log d'<8.23$. But this contradicts with our assumption $\log d'+0.38>10$.

Next we consider the case $\log d'+0.38\leq10$. So, it follows from Lemma \ref{Lemma2.3} that
\begin{equation}\label{eq3.40}
\log |\Omega| \geq-25.2\cdot10^2\log n\log (n+2).
\end{equation}
Using \eqref{eq3.38} and \eqref{eq3.40} leads to
\begin{equation}\label{eq3.41}
y-1 < 25.2\cdot10^2\log n.
\end{equation}
Combining \eqref{eq3.41} and \eqref{eq3.36}, we have
\[
n<70.99+\dfrac{1}{\log n}+25.2\cdot10^2 < 2591.
\]
So, we obtain upper bounds for $y$ and $n$, i.e.
\begin{equation}\label{eq3.42}
n< 2591 \ \ \textrm{and} \ \ y<19808.
\end{equation}
%From the definition of $\Omega$ and \eqref{eq3.35}, we get the inequality
%\begin{equation}\label{eq3.43}
%\biggm\lvert \dfrac{\log (n+2)}{\log n}-\dfrac{z}{y} \biggm\rvert < \dfrac{1}{2.5y(n+2)^{y-1}\log n}.
%\end{equation}
%Since $y\geq2$ and using the inequality $|\frac{\log(n+2)}{\log n}-\frac{z}{y}|<\frac{1}{2y^2}$, we get
%\[
%2y^2< 2.5y(n+2)^{y-1}\log n.
%\]
%Then, $\frac{z}{y}$ is a convergent in the simple continued fraction expansion to $\log (n+2)/ \log n$.
%On the other hand, if $\frac{d_r}{e_r}$ is the $r$-th such convergent then
%$$
%\left|\dfrac{\log (n+2)}{\log %n}-\dfrac{d_r}{e_r}\right|>\dfrac{1}{(u_{r+1}+2)e_r^2},
%$$
%where $u_{r+1}$ is the $(r+1)$-st partial quotient to $\frac{\log (n+2)}{\log n}$  (see Khinchin \cite{K}). Put $\frac{z}{y}=\frac{d_r}{e_r}$. Note that $e_r \leq y$ and $d_r\leq z$. It follows that
%\begin{equation}\label{eq3.44}
%u_{r+1}>\dfrac{2.5(n+2)^{y-1}\log n}{y}-2>6
%\end{equation}
%from $n\geq3$ and $y\geq2$.
Finally, running PARI-GP \cite{PARI} program, we see that the equation \eqref{eq3.34} has no solution $(n,x,y,z)$ in the ranges $71\leq n< 2591$, $73\leq y <19808$, $73<z<39616$ and $n\equiv 7 \pmod 8$. Hence the proof of this case is completed.

%*********************************************	
\end{itemize}

\subsection{The case $n\ge 7$ and $z\ge 2n$}
Since $n\ge7$ and $z\ge2n$, we have $z\ge 2n>h(\mathbb{Q}(\sqrt{-4(n+2)}))$. Then using the argument in the proof of the Case \ref{sec:3.1}, the proof is completed.

\subsection{The case $2<n<7$}
As $2<n<7$, the equation \eqref{1.1} turns into the following equations
\begin{equation}\label{P1}
2^x+5^y=3^z
\end{equation}
\begin{equation}\label{P2}
3^x+6^y=4^z
\end{equation}
\begin{equation}\label{P3}
4^x+7^y=5^z
\end{equation}
\begin{equation}\label{P4}
5^x+8^y=6^z
\end{equation}
where $n=3, 4, 5, 6$, respectively. We first consider the equation \eqref{P1}. By \cite[Theorem 3]{Na}, this equation has only the positive integer solutions $(x,y,z)=(1, 2, 3), (2,1,2)$ which is a desired solution. Then we deal with the equation \eqref{P2} which is impossible. Next we consider the equation \eqref{P3}. Using  \cite[Theorem 9]{Na}, we see
that \eqref{P3} has no positive integer solutions. Finally, we consider the equation \eqref{P4} which is impossible.

\subsection{The case $7\le n \le 64$}
Here we have the case $7\le n \le 64$ for the equation \eqref{1.1}. We wrote a short program in PARI-GP \cite{PARI}. Using it, we see that the equation \eqref{1.1} has no integer
$(n,x,y,z)$ solutions where $7\le n\le 64$,  and $z<2n$. Consequently, the proof  is completed.

%###############################
\subsection*{Acknowledgements}
%################################
We would like to thank to the referee for carefully reading our paper and for giving such constructive comments which substantially helped improving the quality of the paper. The second and
third authors would like to thank to Dr. Paul Voutier for useful suggestions about PARI-GP computations and were supported by T\"{U}B\.{I}TAK (the Scientific and Technological Research
Council of Turkey) under Project No: 117F287. The first and fourth authors were supported by NSF of China (Grant No: 11671153).
%##############################%

\end{document}